\documentclass[12pt,reqno]{amsart}
\usepackage[centertags]{amsmath}
\usepackage{mathrsfs}

\makeatletter
\renewcommand*{\@biblabel}[1]{\hfill#1.}
\makeatother

\theoremstyle{plain} 

\newtheorem{thm}{Theorem}

\theoremstyle{definition}

\theoremstyle{remark}

\numberwithin{equation}{section}

\newcommand{\set}[1]{\left\{#1\right\}}
\newcommand{\pd}{\,\partial}
\newcommand{\comment}[1]{}

\begin{document}

\title[]{On dummy variables of structure-preserving transformations}
\author[J C Ndogmo]{J C Ndogmo}

\address[J C Ndogmo]{School of Mathematics\\
University of the
Witwatersrand\\
Private Bag 3, Wits 2050\\
South Africa}
\email{jean-claude.ndogmo@wits.ac.za}
\begin{abstract} A method  is given for obtaining equivalence subgroups of
a family of differential equations from the equivalence group  of
simpler equations of a similar form, but in which the arbitrary
functions specifying the family element depend on fewer variables.
Examples of applications to classical equations are presented, some
of which show how the method can actually be used for a much easier
determination of the equivalence group itself.
\end{abstract}
\keywords{Equivalence group, equivalence subgroups, arbitrary functions}

\subjclass[2000]{34C20, 35A30}

\maketitle

\section{Introduction}
\label{s:intro}

Denote collectively by $\mathscr{A}$ the set of all arbitrary
functions specifying the family element in a collection
$\mathcal{F}$ of differential equations of the form
\begin{equation}
\label{eq:geq} \Delta (x, y_{(n)}; \mathscr{A} )=0,
\end{equation}
where $x= (x^1, \dots, x^p)$ is the set of independent variables,
and $y_{(n)}$ denotes $y$ and all its derivatives up to the order
$n.$  In the most general case, the function $\mathscr{A}$ may
depend on $x,$ $y,$ and the derivatives of $y$ up to a certain order
not exceeding $n,$ but quite often $\mathscr{A}$ is simply a
function of $x,$ or a constant.\par
Let $G$ be the Lie pseudo-group  of point transformations  of the
form
\begin{equation}\label{eq:gtransfo}
x= \varphi (z, w), \qquad y = \psi (z, w),
\end{equation}
where  $z= (z^1, \dots, z^p)$ is the new set of independent
variables, while $w=w(z)$ is the new dependent variable. The group
$G$ is infinite because as explained in a paper by Tresse \cite[P.
11]{tresse}, its elements depend in general on arbitrary functions
and not on arbitrary constants, and a Lie pseudo-group is to be
understood here in the sense of \cite{olv-maurer2, mrz-surv}, i.e.
as the infinite-dimensional counterpart of a local Lie group of
transformations. We say that $G$ is the group of equivalence
transformations of \eqref{eq:geq} if it is the largest Lie
pseudo-group of point transformations that map \eqref{eq:geq} into
an equation of the same form, that is, if in terms of the same
function $\Delta$ appearing in \eqref{eq:geq}, the transformed
equation has an expression of the form
\begin{subequations}\label{eq:trfoeq}
\begin{align}
\Delta (z, w_{(n)}; \mathscr{B} )=0,& \label{eq:trfo1}  \\[-5mm]
\intertext{where \vspace{-2.5mm}} \mathscr{B}  = T\left(\mathscr{A}
\right),& \label{eq:Gc}
\end{align}
\end{subequations}
 for a certain function $T,$ is the new set of arbitrary functions
specifying the family element in the transformed equation.
Equivalently, $G$ will map elements of $\mathcal{F}$ into
$\mathcal{F},$ and when this holds, \eqref{eq:gtransfo} is called an
equivalence transformation of the original equation \eqref{eq:geq}.
By a result of Lie \cite{lieGc}, \eqref{eq:Gc} defines another group
of transformations $G_c$ induced by $G,$ and acting on the arbitrary
functions $\mathscr{A}$ of the original equation. Invariants of the
group $G_c$  are termed invariants of the differential equation
\eqref{eq:geq}. These invariants are functions which depend on the arbitrary
functions $\mathscr{A}$ of the original equation, and which have exactly
the same expression when they are also formed for the transformed
equation, and they play a crucial role in the classification of the
family of equation \cite{schw-pap, schw-bk}.
\par
Early developments in applications of Lie groups for finding
equivalence transformations of a given differential equation (DE)
started in the work of Lie \cite{lie-ovsi1} and were later pursued
by Tresse \cite{tresse} and Ovsiannikov ~\cite{ovsi-bk}. More recent developments  based on Cartan equivalence methods originated in the works of Olver and collaborators \cite{olv-maurer2, olv-fls2}.

   In this paper, we show that  for a given differential equation in which the
arbitrary functions defining the family element are functions of
independent variables alone, the  equivalence group is a subgroup of
a differential equation of a similar form, but in which the
arbitrary functions also depend on the dependent variable and its
derivatives up to a given order. An extension of the theorem to
equations with arbitrary functions depending on a  subset of the set
consisting of all independent variables, and the dependent variable
and its  derivatives  up to a given order is provided. Examples of
applications to a number of classical equations show how to obtain
equivalence subgroups of a more complex family of equations from
those of simpler ones.

\section{Dummy variables of the equivalence group}

If we consider for instance the general linear differential equation
\begin{equation}\label{eq:glin}
y^{(n)}+   a^1 (x)y^{(n-1)} +  a^2 (x)y^{(n-2)} + \dots + a^n(x) y=0
\end{equation}
the arbitrary functions  $\mathscr{A}$ in equation \eqref{eq:geq}
are the coefficients $a^i(x),$ for $i=1, \dots, n.$ Thus in this
case it appears that $\mathscr{A}=\mathscr{A}(x)$ is a function of
the independent variable alone.  This is quite often the case with
various other linear or nonlinear equations. However, more
generally,  $\mathscr{A}$ arises as a function of the independent
variables, and the dependent variable $y$ and its derivatives up to
a certain order.
\begin{thm}\label{th:dumb}
Denote by {\rm (A)} Eq. \eqref{eq:geq} with
$\mathscr{A}=\mathscr{A}(x)$ and by {\rm  (B)} the same equation,
but in which $\mathscr{A}$ depends on $x$, $y$, and the derivatives
of $y$ up to a certain order $s,$ i.e., in which $\mathscr{A}=
\mathscr{A} (x, y_{(s)}).$ Similarly, denote by $G^A$ and $G^B$ the
equivalence groups for  ${\rm (A)}$ and ${\rm (B)},$ respectively.
Then $G^A$ is a subgroup of  $G^B.$
\end{thm}
\begin{proof}
To fix ideas, suppose that the equation has $m$ arbitrary functions
$A_1,\dots, A_m$ specifying the family element in $\mathcal{F}$ (and
collectively denoted by $\mathscr{A}$), and that in {\rm (B)} we
have $A_j= A_j (x, y_{(s)})$ for all $j=1, \dots,m.$  On the other
hand, for each multi-index $J= (j_1, \dots, j_k),$ where $1\leq j_r
\leq p\,$ for $r=1, \dots, k,$ we use the notation
\begin{equation}\label{eq:yder}
    D_{J}\, y \equiv \frac{\pd^{\,k} y(x)}{\pd x^{j_1} \pd x^{j_2} \dots \pd x^{j_k}
    }\,.
\end{equation}
 Suppose now that \eqref{eq:gtransfo} is an equivalence
transformation of ${\rm (A)}$.  To perform the transformation of
${\rm (B)}$ by \eqref{eq:gtransfo}, we may first ignore the
arguments of the function $\mathscr{A},$ in which sense these
arguments can be considered as dummy. Then, for every multi-index
$J$ with $0\leq \# J=k \leq n,$ each partial derivative $D_{J}\, y$
present in the equation is replaced by a function $T_J (z,
w_{(k)}).$  When all such replacements are made, and all remaining
occurrences of independent variables are expressed in terms of $z$
and $w$ by another application of \eqref{eq:gtransfo}, the resulting
equation is the transformed equation \eqref{eq:trfo1}, in which
$\mathscr{B}$ collectively denotes the new functions $B_1, \dots,
B_m.$ According to \eqref{eq:Gc}, each transformed function $B_j$
has an expression of the form
\begin{align}\label{eq:Bj}
B_j &= F_j (z, A_1, \dots, A_m), \qquad (j=1, \dots, m), \quad r_j
\leq n,
\end{align}
in terms of the original functions $A_1, \dots, A_m,$ and for a
certain function $F_j,$  where $F_j$ does not depend explicitly on
the dependent variable $w$ and its derivatives, on the assumption
that \eqref{eq:gtransfo} is an equivalence transformation of ${\rm
(A)}$. To complete the transformation of ${\rm (B)},$ we may now
apply again \eqref{eq:gtransfo}, to transform the arguments of the
functions $A_j,$ and this will transform $A_j (x, y_{(s)})$ into a
composition of $A_j$ and a function of  $z$  and $w_{(s)},$ and the
latter transformation will thus have no effect on the form of the
transformed equation. Consequently,  the transformed version of
${\rm (B)}$ has the same form as ${\rm (B)},$ and this completes the
proof of the Theorem.
\end{proof}

The proof of Theorem \ref{th:dumb} suggests the following immediate
extension to the case where the function $\mathscr{A}$ in {\rm (A)}
may depend only on a subset of the set of all variables and the
derivatives of $y$ up to a given order.

\begin{thm}\label{th:dumb2}
Denote by {\rm (A)} Eq. \eqref{eq:geq} with
$\mathscr{A}=\mathscr{A}(X),$ where $X$ is a subset of the set
consisting of all independent variables, of the dependent variable
$y$ and its derivatives up to a certain order $r$.  Denote also by
{\rm (B)} the same equation, but in which $\mathscr{A}= \mathscr{A}
(x,  y_{(s)}),$ with $r \leq s.$ Similarly, denote by $G^A$ and
$G^B$ the equivalence groups for ${\rm (A)}$ and ${\rm (B)},$
respectively. Then $G^A$ is a subgroup of  $G^B.$
\end{thm}

\begin{proof}
Using the same notation as in the proof of Theorem \ref{th:dumb},
and reasoning in a similar way, if we apply $G^A$ to transform ${\rm
(B)}$ while first ignoring the arbitrary functions' arguments, the
form of the equation is preserved, and the new arbitrary functions
$B_j$ now take the form
\begin{align}\label{eq:Bj2}
B_j &= F_j (z, w_{(r)}, A_1, \dots, A_m), \qquad (j=1, \dots, m).
\end{align}
We can then transform the arguments  $(x, y_{(s)})$ in  each of the
functions $A_j$ appearing in \eqref{eq:Bj2}, to obtain expressions
of the form $A_j= A_j(z, w_{(s)})$ in \eqref{eq:Bj2}, and the latter
transformation will not affect the form of the equation, as already
seen.
\end{proof}

  It should however be noted that as stated, Theorem \ref{th:dumb}
and Theorem \ref{th:dumb2} holds only under the assumption that in
{\rm (B)}, the function $\mathscr{A}$ depends explicitly on $x=(x^1,
\dots, x^p), \,y,$ and all derivatives of $y$ up to the stated order
$s.$ This is first because when a derivative of $y$ of a given order
$s$ is transformed under a given group, the resulting expression
depends in general on all derivatives of the new dependent variable
$w$ up to the order $s.$ On the other hand, the transformed function
$B_j$ will in general depend on the full set of the new independent
variables $z= (z^1, \dots, z^p)$, and not only on a subset of this
set. Thus, if in {\rm (B)} the function $\mathscr{A}$ does not
explicitly depend on all independent variables or on all derivatives
of $y$ up to the stated order $s,$ this restriction should be
imposed on $G^A,$ and this will yield in general a smaller
equivalence subgroup for {\rm (B)}.\par

   On the other hand,  the converse is not true in Theorem
\ref{th:dumb}, and this is simply because if {\rm (A)} is
transformed by an equivalence transformation of {\rm (B)}, there is
no guarantee that in the transformed equation the functions $B_j$
will not depend on $w$ or its derivatives. Similarly, the converse
to Theorem ~\ref{th:dumb2} is not true.\par

\section{Application to classical equations}

\subsection{The general linear homogeneous ODE}
In the sequel, we shall use the notation $f_a = \pd f/ \pd a,$ for
every function $f$ with argument $a,$ hence subscripts in such
functions will denote differentiation. We shall also use the
notation $\pd_a= \frac{\pd}{\pd a}$ for every variable $a, $  and a
transformed equation will  be assumed to be expanded as a polynomial
in the dependent variable and its derivatives.\par

If we consider for instance \eqref{eq:glin}, its equivalence
transformations are given by
\begin{equation}\label{eq:ET-glin}
x= S(z), \qquad y= L(z) w,
\end{equation}
%
where $S$ and $L$ are arbitrary functions. It follows from Theorem
\ref{th:dumb} that if we assume in \eqref{eq:glin} that the
coefficients $a^i$ depends also on  $y$ and its derivatives up to a
certain order $s,$ then \eqref{eq:ET-glin}  remains an equivalence
transformation of the resulting equation. Equations of the latter
form frequently appear with $s=0,$ so that $a^i= a^i (x,y),$ which
gives rise to a nonlinear equation of the form
\begin{equation}\label{eq:gliny}
y^{(n)}+   a^1 (x,y)y^{(n-1)} +  a^2 (x,y)y^{(n-2)} + \dots +
a^n(x,y) y=0.
\end{equation}
After some calculations, the equivalence group of \eqref{eq:gliny}
is found for $n=3$ to be of the form
\begin{equation}\label{eq:eqv-gliny}
x= S(z), \qquad y= L(z) w+ J(z),
\end{equation}
where $S, L,$ and $J$ are arbitrary functions. It can be shown by
induction that \eqref{eq:eqv-gliny} also holds for all $n \geq 3$ in
\eqref{eq:gliny}
.  Theorem \ref{th:dumb2} states that for any other variant of
\eqref{eq:gliny} in which the coefficients $a^i$ are of the form
$a^i= a^i(X),$ where $X \subset \set{x, y},$ the equivalence group
must be of the predefined form \eqref{eq:eqv-gliny}, where the
functions $S, L,$ and $J$ are to be specified. Finding the
equivalence group of an equation in a more specific form naturally
greatly simplifies calculations.  If for instance we suppose that
$a^i(X)= a^i(y),$ for all $i,$ so that \eqref{eq:gliny} reduces to
\begin{equation}\label{eq:glin0y}
y^{(n)}+   a^1 (y)y^{(n-1)} +  a^2 (y)y^{(n-2)} + \dots + a^n(y)
y=0,
\end{equation}
then since the $a^i$ depend on $y$ alone, $L$ and $J$ must be
constant functions, so that the transformation of $y$ reduces to $y=
k_3\, w + k_4,$ for some constants $k_3$ and $k_4.$ For $n=3,$ under
the corresponding transformations of $x$ and $y$,
 the term not involving the dependent variable
$w$ and its derivatives as a factor in the transformation of
\eqref{eq:glin0y} is $(k_4 a^3 S_z^3)/  k_3$. Since this term may
however depend on $w$ but not on $z,$ we infer that $S_z$ must be a
constant function, so that  the transformation of the independent
variable $x$ must also be a linear function of the form $x= k_1 z+
k_2.$ It is readily seen that the set of transformations
\begin{equation}\label{eq:eqv-glin0y}
x= k_1 z+ k_2, \qquad y= k_3 w+ k_4,
\end{equation}
preserves the form of \eqref{eq:glin0y} for $n=3,$ and therefore
defines its equivalence group. It is also readily found by induction
on $n$ that this remains true for all $n \geq 3.$ Using
\eqref{eq:eqv-gliny} as the predefined form of the equivalence group
leads in a similar manner to a much easier determination of the
equivalence group for other possible values of $\small X.$ In
addition, $X$ needs not be the same for all of the coefficients
$a^i,$ and thus to apply Theorem \ref{th:dumb2} here we only need
 to have $a^i= a^i(X^i),$ with $X^i \subset
\set{x,y}.$ In particular, some of the $a^i$ might be constants.

\subsection{The linear hyperbolic equation}

 Consider the PDE in
two independent variables $t$ and $x$ of the form
\begin{equation}\label{eq:hyper}
u_{tx} + a^1(t,x) u_t + a^2(t,x) u_x + a^3(t,x)u=0,
\end{equation}
generally referred to as the linear hyperbolic second-order
equation. It is well known (see e.g. \cite{ibra-lap, waf-joint})
that the group $G$ of equivalence transformations of this equation
is given in terms of new independent variables $y$ and $z,$ and
dependent variables $w,$ by invertible transformations of the form
\begin{equation}\label{eq:eqv-hyper}
t= R(y), \quad x= S(z), \quad \text{ and }  u= L(y,z) w,
\end{equation}
where $R,\, S,$ and $L$ are arbitrary functions satisfying the
non-vanishing Jacobian condition
\begin{equation}\label{eq:jacohyp1}
R_y S_z L \neq 0,
\end{equation}
where $R_y = \pd_y R$ and $S_z= \pd_z S.$ Consider now a nonlinear
extension of \eqref{eq:hyper} of the form
\begin{equation}\label{eq:hyperu}
u_{tx} + a^1(t,x,u) u_t + a^2(t,x, u) u_x + a^3(t,x, u)u=0,
\end{equation}
in which the dependent variable $u$ also appears as argument in the
arbitrary coefficients. We undertake the somewhat lengthy
calculation of the equivalence group of \eqref{eq:hyperu} which is
not available in the literature, to illustrate how the knowledge of
such a group can be utilized for a much easier determination of the
equivalence group of equations of a similar form. The equivalence
group $H$ of \eqref{eq:hyperu} must be sought in the form
\begin{equation}\label{eq:eqv-hyperu}
t= R(y,z,w), \quad x= S(y,z,w), \quad \text{ and }  u= T(y,z,w),
\end{equation}
for some functions $R, S,$ and $T$ to be specified, and which must
satisfy the non-vanishing Jacobian condition
\begin{equation}\label{eq:jacohyp2}
 \left(-R_z S_y+R_y S_z\right) T_w+\left(R_z S_w-R_w S_z\right)
 T_y+\left(-R_y S_w+R_w S_y\right) T_z \neq 0.
\end{equation}
 However, we notice that under \eqref{eq:eqv-hyperu}, $u_t$ is
 transformed into
$$
u_t= \frac{-S_z T_y+S_y T_z+\left(S_w T_z -S_z T_w\right)
w_y+\left(S_y T_w-S_w T_y\right) w_z}{R_z S_y-R_y S_z+\left(R_z
S_w-R_w S_z\right) w_y+\left(R_w S_y-R_y S_w\right) w_z}.
$$
 The occurrence of derivatives of $w$ in the denominator of this
 transformed expression for $u_t$ will give rise in the expression of $u_{tx}$
to undesired terms in $w_{zz}$ and $w_{yy}$ as well as nonlinear
terms of the form $w_z^i
 w_y^j,$ where $i$ and $j$ are some natural numbers, and they should
 therefore disappear. This disappearance  of derivatives is translated into the
 conditions
 \begin{equation}\label{eq:denomf}
 R_z S_w-R_w S_z=0, \qquad  -R_y
S_w+R_w S_y=0,
 \end{equation}
under which the Jacobian condition \eqref{eq:jacohyp2} is reduced to
\begin{equation}\label{eq:jacohyp2b}
\left(-R_z S_y+R_y S_z\right) T_w \neq 0.
\end{equation}
If we suppose that $R_w=0$ and $S_w \neq 0,$ then it follows from
\eqref{eq:denomf} that $R_y= R_z=0,$ which contradicts
\eqref{eq:jacohyp2b}. Thus we have $R_w=0$ if and only if $S_w=0.$
On the other hand, if we assume that $R_w \neq 0,$ then using
\eqref{eq:denomf} we may write
$$
S_y = R_y (S_w / R_w), \qquad R_z= S_z (R_w / S_w),
$$
which also contradicts \eqref{eq:jacohyp2b}, so that $R_w=S_w=0,$
and the resulting expressions in  \eqref{eq:eqv-hyperu} take the
form
\begin{equation}\label{eq:eqv-hyperu2}
t= R(y,z), \quad x= S(y,z), \quad \text{ and }  u= T(y,z,w).
\end{equation}
Under the new change of variables, the terms in $w_{yy}$ and
$w_{zz}$ in the transformed equation are given by
$$
R_z S_z (-R_z S_y+ R_y S_z)T_w w_{yy} \quad \text{ and } \quad R_y
S_y (-R_z S_y+ R_y S_z)T_w w_{zz},
$$
respectively. On account of \eqref{eq:jacohyp2b}, the vanishing of
these terms is reduced to the condition
$$
R_z S_z= R_y S_y=0,
$$
which on account of the Jacobian condition  \eqref{eq:jacohyp1}
readily yields
$$
R_z=S_y=0.
$$
Consequently, \eqref{eq:eqv-hyperu} reduces to
\begin{equation}\label{eq:eqv-hyperu3}
t= R(y), \quad x= S(z), \quad \text{ and }  u= T(y,z,w).
\end{equation}
Under this new set of transformations, \eqref{eq:hyperu} takes the
form
\begin{equation}\label{eq:hyperu3}
\begin{split}
&\frac{w_y w_z T_{w,w}}{T_w}+w_z \left(a^2
R_y+\frac{T_{y,w}}{T_w}\right)+w_y \left(a^1 S_z+\frac{T_{z,w}}{T_w}\right)\\
 &+\frac{a^3 T R_y S_z+a^1 S_z T_y+a^2 R_y T_z+T_{y,z}}{T_w}
 +w_{y,z}=0,
\end{split}
\end{equation}

which shows that $T$ must be linear in $w.$ Consequently, the
required change of variables must be of the form
\begin{equation}\label{eq:eqv-hyperu4}
t= R(y), \quad x= S(z), \quad \text{ and }  u=  L(y,z) w+ J(y,z),
\end{equation}
and \eqref{eq:eqv-hyperu4} transforms \eqref{eq:hyperu} into
\begin{equation}\label{eq:hyperu4}
\begin{split}
& \frac{J a^3 R_y S_z}{L}+\left(\frac{L_z}{L}+a^1 S_z\right)
w_y+\left(\frac{L_y}{L}+a^2 R_y\right) w_z\\
&+\frac{w \left(a^2 L_z R_y+a^1 L_y S_z+L a^3 R_y
S_z+L_{y,z}\right)}{L}+w_{y,z}=0,
\end{split}
\end{equation}
which is of the required form since we may in this case write the
constant term $ J a^3 R_y S_z/ L$ in the transformed equation in the
form $w \left( J a^3 R_y S_z/ (L w) \right).$ Therefore,
\eqref{eq:eqv-hyperu4} defines the equivalence group $H$ of
\eqref{eq:hyperu}. It differs from the equivalence group $G$ of
\eqref{eq:hyper} only by the additional term $J(y,z)$ in the
transformation of $u,$ which has the effect of adding a sort of
nonhomogeneous term to the transformed equation.\par

   Now that the equivalence group \eqref{eq:eqv-hyperu4} of
\eqref{eq:hyperu} is available, we can clearly demonstrate how
Theorem \ref{th:dumb2} can be used for a much easier determination
of equivalence groups for variants of  \eqref{eq:hyperu}, in which
the arbitrary coefficients $a^i$ have arguments of the form $X^i
\subset \set{t, x, u}.$ Indeed, since  by Theorem \ref{th:dumb2} any
such group  is a subgroup of $H,$ its equivalence transformations
must therefore be sought in the form \eqref{eq:eqv-hyperu4}, which
greatly simplifies calculations. We show this by considering two
examples.\par
To begin with, consider a variant of \eqref{eq:hyperu} of the form
\begin{equation}\label{eq:hyperxp}
u_{t x} + a^1(x) u_t + a^2(t) u_x + a^3(t,x)u=0.
\end{equation}
   Since the coefficient $a^3(t,x)$ of $u$ does not involve $u$
itself, we may assume that $J(y,z)=0$ in \eqref{eq:eqv-hyperu4}, and
we thus look for equivalence transformations of \eqref{eq:hyperxp}
in the form
\begin{equation}\label{eq:eqv-hyperxp1}
t= R(y), \quad x= S(z), \quad \text{ and }  u=  L(y,z) w,
\end{equation}
and under which the coefficients $\theta^y$ of $w_y$ and $\theta^z$
of $w_z$ take the form
$$
\theta^y= a^1 S_z + \frac{L_z}{L}, \quad \text{ and } \quad
\theta^z= a^2 R_y + \frac{L_y}{L}.
$$
Since we  have
$$
0= \frac{\pd \,\theta^y}{\pd y}= \frac{\pd \,\theta^z}{\pd z}, \quad
\text{ and } \quad \frac{\pd \,\theta^y}{\pd y}= \frac{\pd
\,\theta^z}{\pd z}=\frac{L L_{y,z} - L_y L_z}{T^2},
$$
  $L$ must satisfy the condition
  $\pd_y (L_z/L)=0,$ and hence  $L=e^{\, f(y)+g(z) },$ for some
arbitrary functions $f$ and $g.$ This new expression for $L$ reduces
\eqref{eq:eqv-hyperxp1} to
\begin{equation}\label{eq:eqv-hyperxp2}
t= R(y), \quad x= S(z), \quad \text{ and }\quad  u=  e^{ f(y) +
g(z)} w,
\end{equation}
and the latter change of variables transforms \eqref{eq:hyperxp}
into
\begin{equation}\label{eq:hyperxp2}
\begin{split}
& w \left(f_y \left(g_z+a^1 S_z\right)+R_y \left(a^2 g_z+a^3
S_z\right)\right)\\
&+\left(g_z+a^1 S_z\right) w_y+\left(f_y+a^2 R_y\right)
w_z+w_{y,z}=0.
\end{split}
\end{equation}
which is of the prescribed form, and this shows that
\eqref{eq:eqv-hyperxp2} represents the equivalence transformations
of \eqref{eq:hyperxp}.\par

For the second example of determination with a variant of
\eqref{eq:hyperu}, consider the equation
\begin{equation}\label{eq:hypertt}
u_{t x} + a^1(t) u_t + a^2(t) u_x + a^3(t,x)u=0,
\end{equation}
in which the coefficient of $u_t$ and $u_x$ depend on $t$ alone.
Here again, the equivalence group is to be sought in the form
\eqref{eq:eqv-hyperxp1}, and under this change of variables, the
coefficients $\theta^y$  of $w_y$ and $\theta^z$ of $w_z$ take the
form
$$
\theta^y= a^1 S_z + \frac{L_z}{L}, \quad \text{ and } \quad
\theta^z= a^2 R_y + \frac{L_y}{L},
$$
where $a^1= a^1(R(y))$  and $a^2=a^2(R(y)),$ and the conditions
$0=\pd_z \theta^y$ and $0 =\pd_z \theta^z$ lead to the differential
equations
\begin{equation}\label{eq:hypttc1}
\frac{-L_z^2+L L_{zz}}{L^2}+a^1 S_{zz}=0, \quad \text{ and } \quad
\frac{-L_y L_z+L L_{y,z}}{L^2}=0.
\end{equation}
It follows from the arbitrariness of the coefficient $a^1$ in the
first equation of \eqref{eq:hypttc1} that $S= k_1 z+ k_2$ for some
constants $k_1$ and $k_2,$ while
$$-L_y L_z+L L_{y,z}=0.$$
The latter equation has solution $L= g(y) e^{z f(y)},$  where
$g\neq0$ and $f$ are arbitrary functions. A substitution of this
expression for $L$ in the second equation of \eqref{eq:hypttc1}
leads to $ g f_y e^{2 z f}=0,$ and hence $f= k_3$ is a constant
function, so that the corresponding change of variables
\eqref{eq:eqv-hyperxp1} now takes the form
\begin{equation}\label{eq:eqvhyptt22}
t= R(y), \quad x= k_1 z + k_2, \quad \text{ and } \quad  u= g(y)
e^{k_3 z} w.
\end{equation}
The transformation of \eqref{eq:hypertt} under \eqref{eq:eqvhyptt22}
yields the equation
\begin{equation}\label{eq:eqvhypttfn}
\begin{split}
& \left(a^1 k_1+L_z/L\right) w_y+\left(a^2
R_y + L_y/L \right) w_z +w_{y,z}\\
&+ w \left(A k_1 L_y+k_1 L a^3 R_y+a^2 L_z R_y+L_{y,z}\right)/L=0
\end{split}
\end{equation}

which is of the required form, and thus \eqref{eq:eqvhyptt22}
represents the group of equivalence transformations of
\eqref{eq:hypertt}.\par

\subsection{Some applications}

   It is worthwhile making some  comments at this point on some
applications of equivalence transformations with regards to the
integration of differential equations. By essence, equivalence
transformations are a key tool by means of which equations can be
classified and then represented in each equivalence class by a
simpler and more tractable canonical form. If we consider for
instance the variant \eqref{eq:hyperxp}  of the linear hyperbolic
equation if follows from the form of the transformed equation
\eqref{eq:hyperxp2} that if in the corresponding equivalence
transformations \eqref{eq:eqv-hyperxp2} we set
\begin{equation*}
f= k_1 - \int a^2(R) R_y dy, \quad \text{ and }\quad g= k_2 - \int
a^1 (S) S_z dz,
\end{equation*}
for some constants $k_1$ and $k_2,$  and then set $R_y S_z=1,$ then
the transformed equation is reduced to
\begin{equation}\label{eq:hyperxp2red3}
b(y,z) w + w_{y,z}=0,\quad \text{ with }\quad  b(y,z)= a^3- a^1 a^2,
\end{equation}
which gives a much simpler canonical form for all equations of the
form \eqref{eq:hyperxp}. In particular, if in the latter equation we
have
\begin{equation}\label{eq:hyperxpcd1}
    a^3= a^1 a^2,
\end{equation}
then \eqref{eq:hyperxp} is always reducible to the linear wave
equation
\begin{equation}\label{eq:hyperxptrv}
    w_{y,z}=0.
\end{equation}
The complete classification of families of equations can be realized
by means of invariant functions of these equations, which as already
mentioned are defined by the equivalence transformations. For
instance using the functions
$$
H= a^1_t + a^1 a^2 - a^3,  \qquad  K=  a^2_x + a^1 a^2 - a^3
$$
known as Laplace invariants \cite{ibra-lap}, Ovsiannikov obtained
the contact invariants $P= H/K$ and $Q= (\ln |H|)_{t,x}/H$ and used
them to achieve the classification of a subfamily of the linear
hyperbolic equation \eqref{eq:hyperxp} \cite{ovsi-chasyv1, ovsi-bk}.
Similar classifications of differential equations based on invariant
functions have also been carried out by other more recent writers on
the topic \cite{schw-pap, schw-bk, mrz-clhyper, ndogvf}.  Using
Cartan's equivalence method in the form developed by Fels and Olver
\cite{olv-fls2}, Morozov \cite{mrz-clhyper} gave a complete solution
to the equivalence problem for a number of families of partial
differential equations, under the pseudo-group of contact
transformations. \par

Significant simplifications of the equation, which may include a
reduction of its order, are often achieved by the simple vanishing
of the invariants of the equation. For instance, the condition $a^3=
a^1 a^2$ obtained in \eqref{eq:hyperxpcd1} by a direct analysis of
the equivalence transformations of \eqref{eq:hyperxp} and which have
led to the reduction of this equation to the linear wave equation
$w_{y,z}=0,$ actually corresponds exactly to the vanishing of the
invariant $P= H/K$ of this equation. Similar reductions or
derivations of integrability properties  of differential equations
based on the vanishing of invariant functions were carried out by E.
H. Laguerre \cite{lag}, followed by F. Brioschi \cite{brio} for
third and fourth order ODEs, as well as by Liouville \cite{liouv}
and many others.\par

\section{Concluding remarks}

   In this paper, we considered two families ${\rm (A)}$ and ${\rm
(B)}$ of DEs of a similar form depending on arbitrary functions,
where the arguments of each of the arbitrary functions in ${\rm
(A)}$ are a subset of that for the corresponding arbitrary function
in ${\rm (B)}.$ We then showed in Theorem \ref{th:dumb2} that the
equivalence group $G^A$ of $\rm (A)$ must be a subgroup of the
equivalence group $G^B$ of $\rm (B).$ We then showed through some examples of applications how the theorem can be used either for an easier determination of equivalence subgroups
or just equivalence transformations of a complex family of DEs, or
for finding the equivalence group $G^A$ of type ${\rm (A)}$ DEs when
$G^B$ is known.


\end{document}